\documentclass{article}
\usepackage{graphicx,verbatim, bussproofs,tikz,float, latexsym}
\usetikzlibrary{arrows}
\usepackage{pgfplots,multicol}
\usepackage{amsmath,amsthm}
\usepackage{amssymb}
\usepackage{stmaryrd}
\usepackage{proof}
\usepackage{cmll}
\usepackage{bussproofs}
\usepackage{authblk}
\DeclareSymbolFont{extraup}{U}{zavm}{m}{n}
\DeclareMathSymbol{\vardiamond}{\mathalpha}{extraup}{87}

\newcommand{\nomi}{\mathbf{i}}
\newcommand{\nomj}{\mathbf{j}}
\newcommand{\nomk}{\mathbf{k}}

\newcommand{\bigamp}{\mathop{\mbox{\Large \&}}}

\renewcommand{\phi}{\varphi}
\renewcommand{\emptyset}{\varnothing}

\newcommand{\Diamondblack}{\vardiamond}

\renewcommand{\epsilon}{\varepsilon}
\theoremstyle{definition}
\newtheorem{theorem}{Theorem}[section]
\newtheorem{lemma}[theorem]{Lemma}

\newtheorem{proposition}[theorem]{Proposition}

\newtheorem{example}[theorem]{Example}

\newtheorem{corollary}[theorem]{Corollary}

\newtheorem{definition}[theorem]{Definition}
\newtheorem{remark}[theorem]{Remark}

\title{Taming ``McKinsey-like'' formula: An Extended Correspondence and Completeness Theory for Hybrid Logic H(@)}
\author{Zhiguang Zhao}

\date{}

\begin{document}
\maketitle
\begin{abstract}
In the present article, we extend the fragment of inductive formulas for the hybrid language $\mathcal{L}(@)$ in \cite{ConRob} including a McKinsey-like formula, and show that every formula in the extended class has a first-order correspondent, by modifying the algorithm $\mathsf{hybrid}$-$\mathsf{ALBA}$ in \cite{ConRob}. We also identify a subclass of this extended inductive fragment, namely the extended skeletal formulas, which extend the class of skeletal formulas in \cite{ConRob}, each formula in which axiomatize a complete hybrid logic. Our proof method here is proof-theoretic, following \cite{GoVa01,Wa07} and \cite[Chapter 14]{BeBlWo06}, in contrast to the algebraic proof in \cite{ConRob}.

\end{abstract}

\section{Introduction}
\paragraph{Hybrid logic} Hybrid logics \cite[Chapter 14]{BeBlWo06} refer to the class of logics which has higher expressive power than modal logic where a special class of propositional symbols called \emph{nominals} are used to refer to single states. A nominal is true at exactly one world. Different logical connectives are used increase the expressive power, e.g.\ the \emph{satisfaction operator} $@_{\mathbf{i}}\phi$ which means $\phi$ is true at the state denoted by $\nomi$. We use $\mathcal{L}(@)$ to denote the hybrid language with nominals and satisfaction operators, and $\mathbf{K}_{\mathcal{H}(@)}$ to denote the basic Hilbert system.

\paragraph{Correspondence and completeness theory for hybrid logic} In modal logic, a modal formula corresponds to a first-order formula if they are valid on the same class of frames. Sahlqvist \cite{Sa75} and van Benthem \cite{vB83} identified a class of modal formulas (later called \emph{Sahlqvist formulas}) which have first-order correspondents and axiomatize strongly complete normal modal logics with respect to the class of frames defined by them. 

There are extensitve research on the correspondence and completeness theory for hybrid logic, see \cite{BetCMaVi04,Co09,
CoGoVa06b,ConRob,GaGo93,GoVa01,Ho06,HoPa10,Ta05,tCMaVi06,Wa07,Zh21c,Zh22b}. Gargov and Goranko \cite{GaGo93} proved that any extension of $\mathbf{K}_{\mathcal{H}}$ with pure axioms (formulas that contain no ordinary propositional variables but possibly contain nominals) is strongly complete. ten Cate and Blackburn \cite{tCBl06} proved that any pure extensions of $\mathbf{K}_{\mathcal{H}(@)}$ and $\mathbf{K}_{\mathcal{H}(@,\downarrow)}$ are strongly complete. ten Cate, Marx and Viana \cite{tCMaVi06} proved that any extensions of $\mathbf{K}_{\mathcal{H}(@)}$ with modal Sahlqvist formulas (with no nominals but possibly with propositional variables) are strongly complete, and that these two kinds of results cannot be combined in general, since there is a pure formula and a modal Sahlqvist formula which together axiomatize a Kripke-incomplete logic when added to $\mathbf{K}_{\mathcal{H}(@)}$. Conradie and Robinson \cite{ConRob} studied to what extent can these two results be combined in $\mathcal{L}(@)$, using algorithmic and algebraic method. Zhao \cite{Zh21c,Zh22b} studies the correspondence and completeness theory for $\mathcal{L}(@,\downarrow)$ following the algorithmic methodology of \cite{ConRob}.

\paragraph{Our contribution} In the present article, we extend the fragment of inductive formulas in \cite{ConRob} to \emph{extended inductive formulas} for $\mathcal{L}(@)$, which contains a Mckinsey-like formula, and show that every extended inductive formula has a first-order correspondent, by modifying the Ackermann Lemma Based Algorithm $\mathsf{hybrid}$-$\mathsf{ALBA}$ in \cite{ConRob} into another algorithm called $\mathsf{ALBA}^{@}$. We also identify a subclass of this extended inductive fragment, namely the \emph{extended skeletal formulas}, which extend the class of skeletal formulas in \cite{ConRob}, each formula in which axiomatize a complete hybrid logic. Our proof method here is proof-theoretic, following \cite{GoVa01,Wa07,Zh22b} and \cite[Section 3.1 in Chapter 14]{BeBlWo06}, in contrast to the algebraic proof in \cite{ConRob}. We define the class of extended skeletal formulas for $\mathcal{L}(@)$, show that for every such formula $\phi$ and its hybrid pure correspondence $\pi$, $\mathbf{K}_{\mathcal{H}(@)}+\phi$ proves $\pi$, therefore $\mathbf{K}_{\mathcal{H}(@)}+\phi$ is complete with respect to the class of frames defined by $\pi$, using a restricted version $\mathsf{ALBA}^{@}_{\mathsf{Restricted}}$ of the $\mathsf{ALBA}^{@}$.

\paragraph{Structure of the article} The structure of the article is as follows: Section \ref{dSec:Prelim} gives preliminaries on hybrid logic $\mathcal{L}(@)$, including its syntax, semantics and basic Hilbert system $\mathbf{K}_{\mathcal{H}(@)}$. Section \ref{dSec:Prelim:ALBA} lists ingredients on the algorithm and the completeness proof, and defines the expanded hybrid modal language $\mathcal{L}(@)^{+}$, the first-order correspondence language and the standard translation. Section \ref{dSec:Sahl} defines extended inductive formulas and extended skeletal formulas. Section \ref{dSec:ALBA} gives the algorithm $\mathsf{ALBA}^{@}$ and $\mathsf{ALBA}^{@}_{\mathsf{Restricted}}$ for $\mathcal{L}(@)$. Section \ref{dSec:Soundness} shows the soundness of the algorithm $\mathsf{ALBA}^{@}$. Section \ref{dSec:Success} sketches the proof that $\mathsf{ALBA}^{@}$ succeeds on extended inductive formulas and $\mathsf{ALBA}^{@}_{\mathsf{Restricted}}$ succeeds on extended skeletal formulas. Section \ref{dSec:Completeness} proves that $\mathbf{K}_{\mathcal{H}(@)}$ extended with extended skeletal formulas are strongly complete.

\section{Preliminaries on hybrid logic $\mathcal{L}(@)$}\label{dSec:Prelim}

In this section, we give preliminaries on the hybrid logic $\mathcal{L}(@)$ in the style of \cite{Zh21c}. For more details of hybrid logic, see \cite[Chapter 14]{BeBlWo06} and \cite{tC05}.

\subsection{Language and syntax}\label{dSubsec:Lan:Syn}

\begin{definition}
Given a countably infinite set $\mathsf{Prop}$ of propositional variables and a countably infinite set $\mathsf{Nom}$ of nominals which are disjoint, the hybrid language $\mathcal{L}(@)$ is defined as follows:
$$\phi::=p \mid \mathbf{i} \mid \bot \mid \top \mid \neg\phi \mid \phi\lor\phi \mid \phi\land\phi \mid \phi\to\phi \mid \Diamond\phi \mid \Box\phi \mid @_{\mathbf{i}}\phi,$$
where $p\in \mathsf{Prop}$, $\mathbf{i}\in\mathsf{Nom}$. We define $\phi\leftrightarrow\psi:=(\phi\to\psi)\land(\psi\to\phi)$. We say that a formula is \emph{pure} if it contains no propositional variables. We use $\sigma$ to denote a \emph{sorted substitution} that uniformly replaces propositional variables by formulas and nominals by nominals.

We will use \emph{inequalities} of the form $\phi\leq\psi$, where $\phi$ and $\psi$ are formulas, and \emph{quasi-inequalities} of the form $\phi_1\leq\psi_1\ \&\ \ldots\ \&\ \phi_n\leq\psi_n\ \Rightarrow\ \phi\leq\psi$. We will find it easy to work with inequalities $\phi\leq\psi$ in place of implicative formulas $\phi\to\psi$ in Section \ref{dSec:Sahl}.
\end{definition}

\subsection{Semantics}\label{dSubsec:Seman}

\begin{definition}
A \emph{frame} is a pair $\mathbb{F}=(W,R)$ where $W\neq\emptyset$ is the \emph{domain} of $\mathbb{F}$, $R\subseteq W\times W$ is the \emph{accessibility relation}. A \emph{model} is a pair $\mathbb{M}=(\mathbb{F},V)$ where $V:\mathsf{Prop}\cup\mathsf{Nom}\to P(W)$ is a \emph{valuation} on $\mathbb{F}$ such that $V(\nomi)\subseteq W$ is a singleton for all nominals $\nomi\in\mathsf{Nom}$.

The satisfaction relation is given as follows: for any model $\mathbb{M}=(W,R,V)$, any $w\in W$, 

\begin{center}
\begin{tabular}{l c l}
$\mathbb{M},w\Vdash p$ & iff & $w\in V(p)$;\\
$\mathbb{M},w\Vdash\nomi$ & iff & $\{w\}=V(\nomi)$;\\
$\mathbb{M},w\Vdash \bot$ & : & never;\\
$\mathbb{M},w\Vdash \top$ & : & always;\\
$\mathbb{M},w\Vdash \neg\phi$ & iff & $\mathbb{M},w\nVdash\phi$;\\
$\mathbb{M},w\Vdash\phi\lor\psi$ & iff & $\mathbb{M},w\Vdash \phi$ or $\mathbb{M},w\Vdash\psi$;\\
$\mathbb{M},w\Vdash\phi\land\psi$ & iff & $\mathbb{M},w\Vdash \phi$ and $\mathbb{M},w\Vdash\psi$;\\
$\mathbb{M},w\Vdash\phi\to\psi$ & iff & $\mathbb{M},w\nVdash \phi$ or $\mathbb{M},w\Vdash\psi$;\\
$\mathbb{M},w\Vdash\Diamond\phi$ & iff & $\exists v(Rwv\ \mbox{ and }\ \mathbb{M},v\Vdash\phi)$;\\
$\mathbb{M},w\Vdash \Box\phi$ & iff & $\forall v(Rwv\ \Rightarrow\ \mathbb{M},v\Vdash\phi)$;\\
$\mathbb{M},w\Vdash @_{\mathbf{i}}\phi$ & iff & $\mathbb{M},V(\nomi)\Vdash\phi$.
\label{dpage:downarrow}
\end{tabular}
\end{center}

For any formula $\phi$, 
\begin{itemize}
\item $V(\phi):=\{w\in W\mid \mathbb{M},w\Vdash\phi\}$ denotes the \emph{truth set} of $\phi$ in $\mathbb{M}$. 
\item $\phi$ is \emph{globally true} on $\mathbb{M}$ (notation: $\mathbb{M}\Vdash\phi$) if $\mathbb{M},w\Vdash\phi$ for every $w\in W$. 
\item $\phi$ is \emph{valid} on a frame $\mathbb{F}$ (notation: $\mathbb{F}\Vdash\phi$) if $\phi$ is globally true on $(\mathbb{F},V)$ for each valuation $V$.
\end{itemize}
\end{definition}

The semantics of inequalities and quasi-inequalities are given as follows:

\begin{itemize}
\item $\mathbb{M}\Vdash\phi\leq\psi\mbox{ iff }(\mbox{for all }w\in W, \mbox{ if }\mathbb{M},w\Vdash\phi, \mbox{ then }\mathbb{M},w\Vdash\psi).$
\item $\mathbb{M}\Vdash\phi_1\leq\psi_1\ \&\ \ldots\ \&\ \phi_n\leq\psi_n\ \Rightarrow\ \phi\leq\psi\mbox{ iff }$
$$(\mathbb{M}\Vdash\phi_i\leq\psi_i\mbox{ for all }1\leq i\leq n)\mbox{ implies }(\mathbb{M}\Vdash\phi\leq\psi).$$
\end{itemize}

Validities for inequalities and quasi-inequalities are defined similar to formulas. It is obvious that $\mathbb{M}\Vdash\phi\leq\psi$ iff $\mathbb{M}\Vdash\phi\to\psi$.

\subsection{Hilbert system}

The Hilbert system $\mathbf{K}_{\mathcal{H}(@)}$ of $\mathcal{L}(@)$ is given as follows (see \cite{tC05}):
\begin{itemize}
\item[(CT)] $\vdash\phi$ for all classical tautologies $\phi$
\item[(Dual)] $\vdash\Diamond p\leftrightarrow\neg\Box\neg p$
\item[(K)] $\vdash\Box(p\to q)\to(\Box p\to\Box q)$ 
\item[(K$_{@}$)] $\vdash @_{\nomi}(p\to q)\to(@_{\nomi} p\to@_{\nomi} q)$ 
\item[(Selfdual)] $\vdash\neg@_\nomi p\leftrightarrow@_{\nomi}\neg p$
\item[(Ref)] $\vdash@_{\nomi}\nomi$
\item[(Intro)] $\vdash\nomi\land p\to@_{\nomi}p$
\item[(Back)] $\vdash\Diamond@_{\nomi}p\rightarrow @_{\nomi}p$
\item[(Agree)] $\vdash@_\nomi@_\nomj p\to@_{\nomj}p$
\item[(MP)] If $\vdash\phi\to\psi$ and $\vdash\phi$ then $\vdash\psi$
\item[(SB)] If $\vdash\phi$ then $\vdash\sigma(\phi)$
\item[(Nec)] If $\vdash\phi$ then $\vdash\Box\phi$
\item[(Nec$_{@}$)] If $\vdash\phi$ then $\vdash @_{\nomi}\phi$
\item[(Name$_{@}$)] If $\vdash@_{\nomi}\phi$ then $\vdash\phi$, for $\nomi$ not occurring in $\phi$
\item[(BG$_{@}$)] If $\vdash@_{\nomi}\Diamond\nomj\land@_{\nomj}\phi\to\psi$, then $\vdash@_{\nomi}\Diamond\phi\to\psi$, for $\nomj\neq\nomi$ and $\nomj$ not occurring in $\phi$ and $\psi$
\end{itemize}

We use $\mathbf{K}_{\mathcal{H}(@)}+\Sigma$ to denote the system containing all axioms of $\mathbf{K}_{\mathcal{H}(@)}$ and $\Sigma$ and is closed under the rules of $\mathbf{K}_{\mathcal{H}(@)}$. $\vdash_{\Sigma}\phi$ means $\phi$ is a theorem of $\mathbf{K}_{\mathcal{H}(@)}+\Sigma$. When $\Sigma$ is empty, we use $\vdash\phi$ instead of $\vdash_{\emptyset}\phi$ .

$\Gamma\vdash_{\Sigma}\phi$ means that there are $\gamma_1,\ldots,\gamma_n\in\Gamma$ such that $\vdash_{\Sigma}\gamma_1\land\ldots\land\gamma_n\to\phi$. Given a frame class $\mathcal{F}$, we use  $\Gamma\Vdash_{\mathcal{F}}\phi$ to mean that for any frame $\mathbb{F}=(W,R)\in\mathcal{F}$, any valuation $V$ on $\mathbb{F}$, any point $w\in W$, if $\mathbb{F},V,w\Vdash\gamma$ for all $\gamma\in\Gamma$, then $\mathbb{F},V,w\Vdash\phi$.

The following is a theorem for any $\mathbf{K}_{\mathcal{H}(@)}+\Sigma$ where $@_{\nomj}\alpha$ has a single positive occurrence in $\theta$. This will be useful in Section \ref{dSec:Completeness}:\\

$\vdash_{\Sigma}@_{\nomi}\theta(@_{\nomj}\alpha)\leftrightarrow@_{\nomi}\theta(\bot)\lor(@_{\nomi}\theta(\top)\land@_{\nomj}\alpha).$\label{dpage:decomposition:Equivalence}
\begin{definition}[Soundness and Strong Completeness]
\begin{itemize}
\item $\mathbf{K}_{\mathcal{H}(@)}+\Sigma$ is said to be \emph{sound} with respect to $\mathcal{F}$, if $\Gamma\vdash_{\Sigma}\phi$ implies that $\Gamma\Vdash_{\mathcal{F}}\phi$. 
\item $\mathbf{K}_{\mathcal{H}(@)}+\Sigma$ is said to be \emph{strongly complete} with respect to $\mathcal{F}$, if $\Gamma\Vdash_{\mathcal{F}}\phi$ implies that $\Gamma\vdash_{\Sigma}\phi$.
\end{itemize}
\end{definition}

\begin{theorem}[Corollary in \cite{tC05}]\label{dCompleteness:Pure}
If $\Sigma$ is a set of pure $\mathcal{L}(@)$-formulas, then $\mathbf{K}_{\mathcal{H}(@)}+\Sigma$ is sound and strongly complete with respect to the class of frames defined by $\Sigma$.
\end{theorem}

\section{Ingredients of algorithmic correspondence}\label{dSec:Prelim:ALBA}

In this article, we give a modified version $\mathsf{ALBA}^{@}$ of the correspondence algorithm $\mathsf{hybrid}$-$\mathsf{ALBA}$ defined in \cite{ConRob} for $\mathcal{L}(@)$. The algorithm $\mathsf{ALBA}^{@}$ transforms the input extended inductive hybrid formula $\phi\to\psi$ into an equivalent set of pure quasi-inequalities without propositional variables. This section follows the style of \cite{Zh21b}.

The ingredients of algorithmic correspondence can be listed as follows:
\begin{itemize}
\item An expanded hybrid modal language $\mathcal{L}(@)^{+}$, its semantics and its standard translation into first-order correspondence language;
\item An algorithm $\mathsf{ALBA}^{@}$ which transforms a given extended inductive hybrid formula $\phi\to\psi$ into equivalent pure quasi-inequalities $\mathsf{Pure}(\phi\to\psi)$ as well as the first-order correspondent $\mathsf{FO}(\phi\to\psi)$;
\item The class of \emph{extended inductive formulas}, which is strictly larger than inductive formulas defined in \cite{ConRob}, on which $\mathsf{ALBA}^@$ is successful;
\item A soundness proof of the algorithm.
\end{itemize}

The ingredients of the completeness proof can be listed as follows:
\begin{itemize}
\item A translation of the inequalities and quasi-inequalities involved in a restricted version $\mathsf{ALBA}^{@}_{\mathsf{Restricted}}$ of the algorithm $\mathsf{ALBA}^{@}$ into $\mathcal{L}(@)$-formulas;
\item A proof that for any extended skeletal formula $\phi\to\psi$, for each step of the execution of $\mathsf{ALBA}^{@}_{\mathsf{Restricted}}$, the translations of the resulting quasi-inequalities are provable in $\mathbf{K}_{\mathcal{H}(@)}+(\phi\to\psi)$, therefore $\pi$ is provable in $\mathbf{K}_{\mathcal{H}(@)}+(\phi\to\psi)$.
\end{itemize}

In the algorithmic correspondence part, we define an expanded hybrid modal language $\mathcal{L}(@)^{+}$ which the $\mathsf{ALBA}^{@}$ will manipulate (Section \ref{dSub:expanded:language}), the first-order correspondence language and the standard translation (Section \ref{dSub:FOL:ST}). We define the extended inductive/skeletal formulas (Section \ref{dSec:Sahl}), define the algorithms $\mathsf{ALBA}^{@}$ and $\mathsf{ALBA}^{@}_{\mathsf{Restricted}}$ (Section \ref{dSec:ALBA}), show their soundness (Section \ref{dSec:Soundness}) and success on extended inductive/skeletal formulas (Section \ref{dSec:Success}).

In the completeness proof part, we give a translation of the inequalities and quasi-inequalities involved in $\mathsf{ALBA}^{@}_{\mathsf{Restricted}}$ into $\mathcal{L}(@)$ and prove that for any extended skeletal formula $\phi\to\psi$, for each step of the execution of $\mathsf{ALBA}^{@}_{\mathsf{Restricted}}$, the translations of the resulting quasi-inequalities are provable in $\mathbf{K}_{\mathcal{H}(@)}+(\phi\to\psi)$ (Section \ref{dSec:Completeness}).

\subsection{The expanded hybrid modal language $\mathcal{L}(@)^{+}$}\label{dSub:expanded:language}

We define the expanded hybrid modal language $\mathcal{L}(@)^{+}$ used in $\mathsf{ALBA}^{@}$:
$$\phi::=p \mid \nomi \mid \bot \mid \top \mid \neg\phi \mid \phi\land\phi \mid \phi\lor\phi \mid \phi\to\phi \mid \Box\phi \mid \Diamond\phi \mid @_{\nomi}\phi \mid \blacksquare\phi \mid \Diamondblack\phi$$

For $\blacksquare$ and $\Diamondblack$, they are interpreted as the box and diamond modality on the inverse relation $R^{-1}$. For the semantics, the additional connectives are interpreted as follows:
\begin{center}
\begin{tabular}{l c l}
$\mathbb{M},w\Vdash \blacksquare\phi$ & iff & $\forall v(Rvw\ \Rightarrow\ \mathbb{M},v\Vdash\phi)$\\
$\mathbb{M},w\Vdash \Diamondblack\phi$ & iff & $\exists v(Rvw\ \mbox{ and }\ \mathbb{M},v\Vdash\phi)$.\\
\end{tabular}
\end{center}

\subsection{The first-order correspondence language and the standard translation}\label{dSub:FOL:ST}

In the first-order correspondence language, we have a binary predicate symbol $R$ corresponding to the accessibility relation, unary predicate symbols $P$ corresponding to each propositional variable $p$, constant symbols $i$ corresponding to each nominal $\nomi$.

\begin{definition}
The standard translation of $\mathcal{L}(@)^{+}$ is as follows:
\begin{itemize}
\item $ST_{x}(p):=Px$;
\item $ST_{x}(\nomi):=x=i$;
\item $ST_{x}(\bot):=x\neq x$;
\item $ST_{x}(\top):=x=x$;
\item $ST_{x}(\neg\phi):=\neg ST_{x}(\phi)$;
\item $ST_{x}(\phi\land\psi):=ST_{x}(\phi)\land ST_{x}(\psi)$;
\item $ST_{x}(\phi\lor\psi):=ST_{x}(\phi)\lor ST_{x}(\psi)$;
\item $ST_{x}(\phi\to\psi):=ST_{x}(\phi)\to ST_{x}(\psi)$;
\item $ST_{x}(\Box\phi):=\forall y(Rxy\to ST_{y}(\phi))$;
\item $ST_{x}(\Diamond\phi):=\exists y(Rxy\land ST_{y}(\phi))$;
\item $ST_{x}(@_{\nomi}\phi):=ST_{i}(\phi)$;
\item $ST_{x}(\blacksquare\phi):=\forall y(Ryx\to ST_{y}(\phi))$;
\item $ST_{x}(\Diamondblack\phi):=\exists y(Ryx\land ST_{y}(\phi))$.
\end{itemize}
\end{definition}

It is obvious that the translation is correct:
\begin{proposition}
For any model $\mathbb{M}$, any $w\in W$ and any $\mathcal{L}(@)^{+}$-formula $\phi$, 
$$\mathbb{M},w\Vdash\phi\mbox{ iff }\mathbb{M}\vDash ST_{x}(\phi)[w].$$
\end{proposition}

For inequalities and quasi-inequalities, the standard translations are given in a global way:
\begin{definition}
\begin{itemize}
\item $ST(\phi\leq\psi):=\forall x(ST_{x}(\phi)\to ST_{x}(\psi))$;
\item $ST(\phi_1\leq\psi_1\ \&\ \ldots\ \&\ \phi_n\leq\psi_n\ \Rightarrow\ \phi\leq\psi):= ST(\phi_1\leq\psi_1)\land\ldots\land ST(\phi_n\leq\psi_n)\to ST(\phi\leq\psi)$.
\end{itemize}
\end{definition}

\begin{proposition}\label{dProp:ST:ineq:quasi:mega}
For any model $\mathbb{M}$, any inequality $\mathsf{Ineq}$, any quasi-inequality $\mathsf{Quasi}$,
$$\mathbb{M}\Vdash\mathsf{Ineq}\mbox{ iff }\mathbb{M}\vDash ST(\mathsf{Ineq});$$
$$\mathbb{M}\Vdash\mathsf{Quasi}\mbox{ iff }\mathbb{M}\vDash ST(\mathsf{Quasi}).$$
\end{proposition}

\section{Extended inductive/skeletal formulas}\label{dSec:Sahl}

In the present section, we use the unified correspondence style definition (cf.\ \cite{CoPa12,CPZ:Trans,PaSoZh16,Zh21b}) to define extended inductive/skeletal formulas. We will find it convenient to use inequalities $\phi\leq\psi$ instead of implicative formulas $\phi\to\psi$. For formulas $\theta$ not of the form $\phi\to\psi$, we treat it as $\top\to\theta$.

\begin{definition}[Order-type](cf.\ \cite[page 346]{CoPa12})
For any $n$-tuple $(p_1, \ldots, p_n)$ of propositional variables, an \emph{order-type} $\epsilon$ of this $n$-tuple is an element in $\{1,\partial\}^{n}$. With respect to an order-type $\epsilon$, we say that $p_i$ has order-type 1 (resp.\ $\partial$) if $\epsilon_i=1$ (resp.\ $\epsilon_i=\partial$), and write $\epsilon(p_i)=1$ (resp.\ $\epsilon(p_i)=\partial$). We use $\epsilon^{\partial}$ to denote the opposite order-type of $\epsilon$ where $\epsilon^{\partial}(p_i)=1$ if $\epsilon(p_i)=\partial$ and $\epsilon^{\partial}(p_i)=\partial$ if $\epsilon(p_i)=1$.
\end{definition}

\begin{definition}[Signed generation tree]\label{dadef: signed gen tree}(cf.\ \cite[Definition 4]{CPZ:Trans})
The \emph{positive}/\emph{negative generation tree} of an $\mathcal{L}(@)^{+}$-formula $\phi$ is defined as follows: we first label the root of the syntactic generation tree of $\phi$ with $+$ (resp.\ $-$), then label the children nodes as below:
\begin{itemize}
\item Label the same sign to the children nodes of a mother node labelled with $\lor, \land, \Box$, $\Diamond$, $\blacksquare$, $\Diamondblack$;
\item Label the opposite sign to the child node of a mother node labelled with $\neg$;
\item Label the opposite sign to the first child node and the same sign to the second child node of a mother node labelled with $\to$;
\item Label the same sign to the second child node of a mother node labelled with $@$ (notice that we do not label the first child node, i.e.\ the nominal node).
\end{itemize}
A node in a signed generation tree is called \emph{positive} (resp.\ \emph{negative}) if it is signed $+$ (resp.\ $-$).
\end{definition}

\begin{example}
The positive generation tree of $+\Box (p \lor \neg\Diamond q)\to\Box q$ is given in Figure \ref{dfig:example}.
\begin{figure}[htb]
\centering
\begin{tikzpicture}
\tikzstyle{level 1}=[level distance=0.8cm, sibling distance=1.5cm]
\tikzstyle{level 2}=[level distance=0.8cm, sibling distance=1.5cm]
\tikzstyle{level 3}=[level distance=0.8cm, sibling distance=1.5cm]
 \node {$+\to$}         
              child{node{$-\Box$}
                     child{node{$-\lor$}
                           child{node{$-p$}}
                            child{node{$-\neg$}
                             child{node{$+\Diamond$}
                                child{node{$+q$}}}}}}
              child{node{$+\Box$}
                        child{node{$+q$}}} 
;
\end{tikzpicture}
\caption{Positive generation tree for $\Box (p \lor \neg\Diamond q)\to\Box q$}
\label{dfig:example}
\end{figure}
\end{example}

We use $+\phi$ and $-\psi$ in $\phi\leq\psi$ when defining extended inductive/skeletal inequalities. For any $\ast\in\{+,-\}$, $\ast\phi$ is \emph{uniform} in $p_i$ if all occurrences of $p_i$ in $\ast\phi$ have the same sign, and $\ast\phi$ is $\epsilon$-\emph{uniform} in an array $\vec{p}$ (notation: $\epsilon(\ast\phi)$) if $\ast\phi$ is uniform in each $p_i$ in $\vec{p}$, where $p_i$ always has the sign $+$ (resp.\ $-$) if $\epsilon(p_i)=1$ (resp.\ $\partial$).

For any $\epsilon$ over $(p_1,\ldots p_n)$, any $\phi$, any $i=1,\ldots,n$, any $\ast\in\{+,-\}$, an \emph{$\epsilon$-critical node} in $\ast\phi$ is a leaf node $+p_i$ when $\epsilon(p_i)=1$ or $-p_i$ when $\epsilon(p_i)=\partial$. An $\epsilon$-{\em critical branch} in $\ast\phi$ is a branch from the root to an $\epsilon$-critical node. The $\epsilon$-critical branches are those which $\mathsf{ALBA}^{@}$ and $\mathsf{ALBA}^{@}_{\mathsf{Restricted}}$ will solve for.

We use $+\psi\prec\ast\phi$ (resp.\ $-\psi\prec\ast\phi$) to indicate that an occurrence of a subformula $\psi$ inherits the positive (resp.\ negative) sign from $\ast\phi$. We use $\epsilon(\gamma) \prec \ast \phi$ (resp.\ $\epsilon^\partial(\gamma) \prec \ast \phi$) to denote that the signed generation subtree $\gamma$, with the sign inherited from $\ast \phi$, is $\epsilon$-uniform (resp.\ $\epsilon^\partial$-uniform). $p$ is \emph{positive} (resp.\ \emph{negative}) in $\phi$ if $+p\prec+\phi$ (resp.\ $-p\prec+\phi$) for all occurrences of $p$ in $\phi$.

\begin{definition}\label{dadef:good:branches}(cf.\ \cite[Definition 5]{CPZ:Trans})
Nodes in signed generation trees are called \emph{outer nodes} and \emph{inner nodes}, according to Table \ref{daJoin:and:Meet:Friendly:Table}. For the names of outer nodes and inner nodes in the classification, see Remark \ref{Remark:Classification}. For the name SRA and SRR, they stands for ``syntactic right adjoint'' and ``syntactic right residual'', respectively. This is based on the algebraic properties of the interpretations of the connectives. For more details, see \cite[Remark 3.24]{PaSoZh16}.

For any $\ast\in\{+,-\}$, a branch in a $\ast\phi$ is called a \emph{extended good branch} if it is the concatenation of three paths $P_1,P_2,P_3$, one of which might be of length $0$, such that $P_1$ is a path from the leaf consisting (apart from variable nodes) of inner nodes only, and $P_2$ consists (apart from variable nodes) of outer nodes only, and in $P_3$, the node (apart from variable nodes) farthest to the root is $@$. A branch is called a \emph{extended skeletal branch} if it is an extended good branch and the path $P_1$ is of length 0. A branch is called a \emph{good branch} if it is an extended good branch and the path $P_3$ is of length 0. A branch is called a \emph{skeletal branch} if it is an extended good branch and the path $P_1,P_3$ are of length 0.
\begin{table}
\begin{center}
\begin{tabular}{| c | c |}
\hline
Outer &Inner\\
\hline
 & SRA \\
&
\begin{tabular}{c c c c c}
$+$ &$\wedge$ &$\Box$ & $\neg$ & $@$\\
$-$ &$\vee$ &$\Diamond$ & $\neg$ & $@$\\
\end{tabular}\\
\hline
&SRR\\
\begin{tabular}{c c c c c c c}
$+$ & $\vee$ & $\wedge$ &$\Diamond$ & $\neg$  & $@$ &\\
$-$ & $\wedge$ & $\vee$ &$\Box$ & $\neg$  & $@$ & $\to$\\
\end{tabular}
&\begin{tabular}{c c c c}
$+$ &$\vee$ &$\to$\\
$-$ & $\wedge$ &\\
\end{tabular}
\\
\hline
\end{tabular}
\end{center}
\caption{Outer and Inner nodes.}\label{daJoin:and:Meet:Friendly:Table}
\vspace{-1em}
\end{table}
\end{definition}

\begin{definition}[Extended Inductive/Extended Skeletal/Inductive/Skeletal Inequalities]\label{daInducive:Ineq:Def}(cf.\ \cite[Definition 6]{CPZ:Trans})
For any order-type $\epsilon$ and any strict partial order $<_\Omega$ on $p_1,\ldots p_n$ (the \emph{dependence order}), the signed generation tree $*\phi$ $(*\in\{-,+\})$ of a formula $\phi(p_1,\ldots p_n)$ is \emph{$(\Omega,\epsilon)$-extended inductive/extended skeletal/inductive/skeletal} if

\begin{enumerate}
\item for any $i=1,\ldots,n$, every $\epsilon$-critical branch with leaf $p_i$ is extended good/extended skeletal/good/skeletal;
\item every SRR-node in an $\epsilon$-critical branch is either $\bigstar(\gamma,\beta)$ or $\bigstar(\beta,\gamma)$, where $\bigstar$ is a binary connective, the $\epsilon$-critical branch goes through $\beta$, and
\begin{enumerate}
\item $\epsilon^\partial(\gamma) \prec \ast \phi$;
\item $p_k <_{\Omega} p_i$ for every $p_k$ occurring in $\gamma$.
\end{enumerate}
\end{enumerate}

An inequality $\phi\leq\psi$ is \emph{$(\Omega,\epsilon)$-extended inductive/extended skeletal/inductive/skeletal} if the signed generation trees $+\phi$ and $-\psi$ are so. An inequality $\phi\leq\psi$ is \emph{extended inductive/extended skeletal/inductive/skeletal} if it is so for some ($\Omega$, $\epsilon$). A formula $\phi\to\psi$ (resp.\ a formula $\theta$ which is not implicative) is \emph{extended inductive/extended skeletal/inductive/skeletal} if $\phi\leq\psi$ (resp.\ $\top\leq\theta$) is so.
\end{definition}

From the definition, it is easy to see that the classes of extended inductive/skeletal formulas are respectively strictly larger than the classes of inductive/skeletal formulas. For the examples distinguish the extended classes and the non-extended classes, we give the following examples:
\begin{example}
For the formula $\Box @_{\nomi}\Diamond\Box p\to\Diamond\Box p$, it is $(\Omega,\epsilon)$-extended inductive where $\epsilon(p)=1$ and $<_{\Omega}$ is empty. Here The critical branch is $+\Box@_{\nomi}\Diamond\Box p$, where the $P_3$-part is $+\Box,+@_{\nomi}$, the $P_2$-part is $+\Diamond$, the $P_1$-part is $+\Box$. It is obvious that it is not an inductive formula.
\end{example}

\begin{example}\label{Example:McKinsey}
For the McKinsey-like formula $\Box @_{\nomi}\Diamond p\to\Diamond\Box p$
, it is $(\Omega,\epsilon)$-extended inductive where $\epsilon(p)=1$ and $<_{\Omega}$ is empty. Here The critical branch is $+\Box@_{\nomi}\Diamond p$, where the $P_3$-part is $+\Box,+@_{\nomi}$, the $P_2$-part is $+\Diamond$, the $P_1$-part is empty. It is obvious that it is not a skeletal formula.
\end{example}

\begin{remark}\label{Remark:Classification}
The classification of outer nodes and inner nodes is based on how different connectives behave in the algorithm. When the input formula is an extended inductive formula, the algorithm first decompose the $P_3$-part, then the outer part of the formula, and then the inner part of the formula, as we will see in the algorithm.

The $P_3$-part is the special point of the definition of extended inductive formulas. The basic idea is that when computing the minimal valuation, the $@$-operator can ``reset'' the minimal valuation, so no matter what connectives occur in the $P_3$-part, the last connective $@$ can make the minimal valuation a nominal again, so the ``outer-inner structure'' is not necessary in the $P_3$-part. From the example of the algorithm, this will be more clear.
\end{remark}

\section{The algorithm $\mathsf{ALBA}^{@}$ and $\mathsf{ALBA}^{@}_{\mathsf{Restricted}}$}\label{dSec:ALBA}

In this section, we define the modified version of the correspondence algorithms $\mathsf{ALBA}^{@}$ and $\mathsf{ALBA}^{@}_{\mathsf{Restricted}}$ for $\mathcal{L}(@)$, following the style of \cite{ConRob}. The major difference between the two algorithms is that the latter algorithm does not treat the $P_1$-part of extended good branches, therefore it could only treat extended skeletal branches and extended skeletal formulas. We define $\mathsf{ALBA}^{@}$, and $\mathsf{ALBA}^{@}_{\mathsf{Restricted}}$ is the algorithm without the Stage 2.3.

The input of $\mathsf{ALBA}^{@}$ is a formula $\phi\to\psi$ (when the input formula $\theta$ is not implicative, we first rewrite it into $\top\to\theta$). $\mathsf{ALBA}^{@}$ transforms it into an inequality $\phi\leq\psi$. Then $\mathsf{ALBA}^{@}$ goes in three steps.

\begin{enumerate}
\item \textbf{First approximation}:
We apply the following \emph{first approximation rule} to $\phi\leq\psi$:
$$\infer{\nomi_0\leq\phi\ \&\ \psi\leq \neg\nomi_1\ \Rightarrow \nomi_0\leq\neg\nomi_1}{\phi\leq\psi}
$$
We call the quasi-inequality $\nomi_0\leq\phi\ \&\ \psi\leq \neg\nomi_1\ \Rightarrow \nomi_0\leq\neg\nomi_1$ a \emph{system}.
\item \textbf{The reduction stage}:
In this stage, for each system $\mathsf{S}\ \Rightarrow\ \nomi_0\leq\neg\nomi_1$ obtained during this stage, we apply the following rules to prepare for eliminating all the proposition variables in $\mathsf{S}$:

\begin{itemize}
\item \textbf{Stage 2.1: Decomposing the $P_3$-part}
\begin{enumerate}
\item Suppose that we have an inquality $\nomi\leq\theta(@_{\nomj}\alpha)$, where $+@_{\nomj}\alpha\prec +\theta$ is a farthest $@$ node in the $P_3$-part, then we have the following decomposition rule:
$$\infer{(\nomi\leq\theta(\bot)\ \&\ \mathsf{S}\ \Rightarrow\ \nomi_0\leq\neg\nomi_1)\ (\nomi\leq\theta(\top)\ \&\ \nomj\leq\alpha\ \&\ \mathsf{S}\ \Rightarrow\ \nomi_0\leq\neg\nomi_1)}{\nomi\leq\theta(@_{\nomj}\alpha)\ \&\ \mathsf{S}\ \Rightarrow\ \nomi_0\leq\neg\nomi_1}
$$
\item Suppose that we have an inquality $\nomi\leq\theta(@_{\nomj}\alpha)$, where $-@_{\nomj}\alpha\prec +\theta$ is a farthest $@$ node in the $P_3$-part, then we have the following decomposition rule:
$$\infer{(\nomi\leq\theta(\top)\ \&\ \mathsf{S}\ \Rightarrow\ \nomi_0\leq\neg\nomi_1)\ (\nomi\leq\theta(\bot)\ \&\ \alpha\leq\neg\nomj\ \&\ \mathsf{S}\ \Rightarrow\ \nomi_0\leq\neg\nomi_1)}{\nomi\leq\theta(@_{\nomj}\alpha)\ \&\ \mathsf{S}\ \Rightarrow\ \nomi_0\leq\neg\nomi_1}
$$
\item Suppose that we have an inquality $\theta(@_{\nomj}\alpha)\leq\neg\nomi$, where $-@_{\nomj}\alpha\prec -\theta$ is a farthest $@$ node in the $P_3$-part, then we have the following decomposition rule:
$$\infer{(\theta(\bot)\leq\neg\nomi\ \&\ \mathsf{S}\ \Rightarrow\ \nomi_0\leq\neg\nomi_1)\ (\theta(\top)\leq\neg\nomi\ \&\ \nomj\leq\alpha\ \&\ \mathsf{S}\ \Rightarrow\ \nomi_0\leq\neg\nomi_1)}{\theta(@_{\nomj}\alpha)\leq\neg\nomi\ \&\ \mathsf{S}\ \Rightarrow\ \nomi_0\leq\neg\nomi_1}
$$
\item Suppose that we have an inquality $\theta(@_{\nomj}\alpha)\leq\neg\nomi$, where $+@_{\nomj}\alpha\prec -\theta$ is a farthest $@$ node in the $P_3$-part, then we have the following decomposition rule:
$$\infer{(\theta(\top)\leq\neg\nomi\ \&\ \mathsf{S}\ \Rightarrow\ \nomi_0\leq\neg\nomi_1)\ (\theta(\bot)\leq\neg\nomi\ \&\ \alpha\leq\neg\nomj\ \&\ \mathsf{S}\ \Rightarrow\ \nomi_0\leq\neg\nomi_1)}{\theta(@_{\nomj}\alpha)\leq\neg\nomi\ \&\ \mathsf{S}\ \Rightarrow\ \nomi_0\leq\neg\nomi_1}
$$
\end{enumerate}

\item \textbf{Stage 2.2: Decomposing the $P_2$-part}

In the current stage, the following rules are applied to decompose the $P_2$-part of critical branches. Except for two splitting rules for $\lor$ and $\land$, the other rules execute on a single inequality in the antecedent part of the quasi-inequalities and rewrite it into one or two inequalities:
\begin{enumerate}
\item Splitting rules:
$$
\infer{(\nomi\leq\beta\ \&\ \mathsf{S}\ \Rightarrow\ \nomi_0\leq\neg\nomi_1)\ (\nomi\leq\gamma\ \&\ \mathsf{S}\ \Rightarrow\ \nomi_0\leq\neg\nomi_1)}{\nomi\leq\beta\lor\gamma\ \&\ \mathsf{S}\ \Rightarrow\ \nomi_0\leq\neg\nomi_1}
$$
$$
\infer{(\beta\leq\neg\nomi\ \&\ \mathsf{S}\ \Rightarrow\ \nomi_0\leq\neg\nomi_1)\ (\gamma\leq\neg\nomi\ \&\ \mathsf{S}\ \Rightarrow\ \nomi_0\leq\neg\nomi_1)}{\beta\land\gamma\leq\neg\nomi\ \&\ \mathsf{S}\ \Rightarrow\ \nomi_0\leq\neg\nomi_1}
$$
$$
\infer{\nomi\leq\beta\ \ \ \nomi\leq\gamma}{\nomi\leq\beta\land\gamma}
\qquad
\infer{\alpha\leq\neg\nomi\ \ \ \beta\leq\neg\nomi}{\alpha\lor\beta\leq\neg\nomi}
$$
\item Approximation rules:
$$
\infer{\nomj\leq\alpha\ \ \ \nomi\leq\Diamond\nomj}{\nomi\leq\Diamond\alpha}
\qquad
\infer{\alpha\leq\neg\nomj\ \ \ \Box\neg\nomj\leq\neg\nomi}{\Box\alpha\leq\neg\nomi}
$$
$$
\infer{\nomj\leq\alpha}{\nomi\leq @_{\nomj}\alpha}
\qquad
\infer{\alpha\leq\neg\nomj}{@_{\nomj}\alpha\leq\neg\nomi}
$$
$$
\infer{\nomj\leq\alpha\ \ \ \ \ \ \ \beta\leq\neg\nomk\ \ \ \ \ \ \ \nomj\rightarrow\neg\nomk\leq\neg\nomi}{\alpha\rightarrow\beta\leq\neg\nomi}
$$
The nominals introduced by the approximation rules must not occur in the system before applying the rule.
\item Residuation rules:
$$
\infer{\alpha\leq\neg\nomi}{\nomi\leq\neg\alpha}
\qquad
\infer{\nomi\leq\alpha}{\neg\alpha\leq\neg\nomi}
$$
\end{enumerate}

\item \textbf{Stage 2.3: Decomposing the $P_1$-part}

In the current stage, the following rules are applied to decompose the $P_1$-part of the critical branch. Except for two residuation rules for $@$, the other rules execute on a single inequality in the antecedent part of the quasi-inequalities and rewrite it into one or two inequalities.
\begin{enumerate}
\item Splitting rules:
$$\infer{\alpha\leq\beta\ \ \ \alpha\leq\gamma}{\alpha\leq\beta\land\gamma}
\qquad
\infer{\alpha\leq\gamma\ \ \ \beta\leq\gamma}{\alpha\lor\beta\leq\gamma}
$$
\item Residuation rules:
$$\infer{\beta\leq\neg\alpha}{\alpha\leq\neg\beta}
\qquad
\infer{\neg\beta\leq\alpha}{\neg\alpha\leq\beta}
\qquad
\infer{\alpha\leq\blacksquare\beta}{\Diamond\alpha\leq\beta}
\qquad
\infer{\Diamondblack\alpha\leq\beta}{\alpha\leq\Box\beta}
$$
$$
\infer{\alpha\leq\beta\to\gamma}{\alpha\land\beta\leq\gamma}
\qquad
\infer{\alpha\land\neg\beta\leq\gamma}{\alpha\leq\beta\lor\gamma}
\qquad
\infer{\alpha\land\beta\leq\gamma}{\alpha\leq\beta\to\gamma}
$$
$$
\infer{\beta\leq\alpha\to\gamma}{\alpha\land\beta\leq\gamma}
\qquad
\infer{\alpha\land\neg\gamma\leq\beta}{\alpha\leq\beta\lor\gamma}
\qquad
\infer{\beta\leq\alpha\to\gamma}{\alpha\leq\beta\to\gamma}
$$
$$\infer{(\alpha\leq\bot\ \&\ \mathsf{S}\ \Rightarrow\ \nomi_0\leq\neg\nomi_1)\ (\nomj\leq\beta\ \&\ \mathsf{S}\ \Rightarrow\ \nomi_0\leq\neg\nomi_1)}{\alpha\leq@_{\nomj}\beta\ \&\ \mathsf{S}\ \Rightarrow\ \nomi_0\leq\neg\nomi_1}
$$
$$
\infer{(\top\leq\beta\ \&\ \mathsf{S}\ \Rightarrow\ \nomi_0\leq\neg\nomi_1)\ (\alpha\leq\neg\nomj\ \&\ \mathsf{S}\ \Rightarrow\ \nomi_0\leq\neg\nomi_1)}{@_{\nomj}\alpha\leq\beta\ \&\ \mathsf{S}\ \Rightarrow\ \nomi_0\leq\neg\nomi_1}
$$
\end{enumerate}

\item \textbf{Stage 2.4: The Ackermann stage}

In the current stage, we compute the minimal/maximal valuations for propositional variables and use the Ackermann rules to eliminate all the propositional variables.
\begin{enumerate}
\item The right-handed Ackermann rule:
$$\infer{\bigamp_{j=1}^{m}\eta_j(\theta/p)\leq\iota_j(\theta/p)\ \Rightarrow\ \nomi_0\leq\neg\nomi_1}{\bigamp_{i=1}^{n}\theta_i\leq p\ \&\ \bigamp_{j=1}^{m}\eta_j\leq\iota_j\ \Rightarrow\ \nomi_0\leq\neg\nomi_1}$$

where:
\begin{enumerate}
\item $p$ does not occur in $\theta_1, \ldots, \theta_n$;
\item Each $\eta_i$ is positive, and each $\iota_i$ negative in $p$, for $1\leq j\leq m$;
\item $\theta:=\theta_1\lor\ldots\lor\theta_n$. When $n=0$, we define $\theta:=\bot$.
\end{enumerate}

\item The left-handed Ackermann rule:
$$\infer{\bigamp_{j=1}^{m}\eta_j(\theta/p)\leq\iota_j(\theta/p)\ \Rightarrow\ \nomi_0\leq\neg\nomi_1}{\bigamp_{i=1}^{n}p\leq\theta_i\ \&\ \bigamp_{j=1}^{m}\eta_j\leq\iota_j\ \Rightarrow\ \nomi_0\leq\neg\nomi_1}$$

where:
\begin{enumerate}
\item $p$ does not occur in $\theta_1, \ldots, \theta_n$;
\item Each $\eta_i$ is negative, and each $\iota_i$ positive in $p$, for $1\leq j\leq m$;
\item $\theta:=\theta_1\land\ldots\land\theta_n$. When $n=0$, we define $\theta:=\top$.
\end{enumerate}
\end{enumerate}
\end{itemize}

\item \textbf{Output}: If in the previous stage, for some system, the algorithm gets stuck, i.e.\ some propositional variables cannot be eliminated by the reduction rules in Stage 2, then the algorithm stops and output ``failure''. Otherwise, the system after the first-approximation has been reduced to a set of pure quasi-inequalties. Then the output is this set of pure quasi-inequalities and the conjunction of universally quantified (over the individual symbols corresponding to nominals) first-order sentences of their standard translations. 
\end{enumerate}

\begin{remark}
\begin{itemize}
\item There are no preprocessing rules in Stage 1 in $\mathsf{ALBA}^{@}$ and $\mathsf{ALBA}^{@}_{\mathsf{Restricted}}$ compared with $\mathsf{hybrid}$-$\mathsf{ALBA}$ in \cite{ConRob}, i.e.\ there are no distribution rules, splitting rules, and monotone/antitone variable elimination rules in Stage 1. This is because we incorporate these rules in the splitting rules for $\land,\lor$ in Stage 2.2 and the Ackermann rules (where we allow empty minimal valuations) in Stage 2.4.
\item A special feature of $\mathsf{ALBA}^{@}_{\mathsf{Restricted}}$ compared with $\mathsf{ALBA}^{@}$ is that there is no expanded hybrid language needed in $\mathsf{ALBA}^{@}_{\mathsf{Restricted}}$, and there is no tense operators needed. Another feature of $\mathsf{ALBA}^{@}_{\mathsf{Restricted}}$ is that during Stage 2, for each inequality involved, they are of the form $\nomi\leq\gamma$ or $\gamma\leq\neg\nomi$, which means that they can be equivalently translated into hybrid formulas of the form $@_{\nomi}\gamma$ or $\neg@_\nomi \gamma$, as we can see in Section \ref{dSec:Success} and \ref{dSec:Completeness}.
\end{itemize}
\end{remark}

\begin{example}
Consider the formula $\Box\Diamond @_{\nomi}\Diamond p\to\Diamond\Box p$ as given in Example \ref{Example:McKinsey}. We first rewrite it into $$\nomi_0\leq\Box\Diamond @_{\nomi}\Diamond p\ \&\ \Diamond\Box p\leq\neg\nomi_1\ \Rightarrow\ \nomi_0\leq\neg\nomi_1.$$

Now we know that for $\Diamond\Box p\leq\neg\nomi_1$, we cannot use it to compute the minimal valuation, as in the McKinsey formula $\Box\Diamond p\to\Diamond\Box p$ in modal logic. 

For the $\nomi_0\leq\Box\Diamond @_{\nomi}\Diamond p$ part, in standard ways to compute the minimal valuation, we will get $\Diamondblack\nomi_0\leq \Diamond @_{\nomi}\Diamond p$ and we get stuck.

However, if we use the decomposition rules for $P_3$-part, then we can ``jump over'' the outer $\Box\Diamond$ part, and get two systems
$$\nomi_0\leq\Box\Diamond\top\ \&\ \nomi\leq\Diamond p\ \&\ \Diamond\Box p\leq\neg\nomi_1\ \Rightarrow\ \nomi_0\leq\neg\nomi_1$$ 
and 
$$\nomi_0\leq\Box\Diamond\bot\ \Diamond\Box p\leq\neg\nomi_1\ \Rightarrow\ \nomi_0\leq\neg\nomi_1.$$ 
For the first system, we have the following execution:

$\nomi_0\leq\Box\Diamond\top\ \&\ \nomi\leq\Diamond p\ \&\ \Diamond\Box p\leq\neg\nomi_1\ \Rightarrow\ \nomi_0\leq\neg\nomi_1$\\
$\nomi_0\leq\Box\Diamond\top\ \&\ \nomi\leq\Diamond \nomj\ \&\ \nomj\leq p\ \&\ \Diamond\Box p\leq\neg\nomi_1\ \Rightarrow\ \nomi_0\leq\neg\nomi_1$\\
$\nomi_0\leq\Box\Diamond\top\ \&\ \nomi\leq\Diamond \nomj\ \&\ \Diamond\Box\nomj\leq\neg\nomi_1\ \Rightarrow\ \nomi_0\leq\neg\nomi_1$.\\

For the second system, we have the following execution:

$\nomi_0\leq\Box\Diamond\bot\ \Diamond\Box p\leq\neg\nomi_1\ \Rightarrow\ \nomi_0\leq\neg\nomi_1$\\
$\nomi_0\leq\Box\Diamond\bot\ \Diamond\Box\bot\leq\neg\nomi_1\ \Rightarrow\ \nomi_0\leq\neg\nomi_1$.\\

We skip the standard translation step.

As we can see, the execution above is also an execution of $\mathsf{ALBA}^{@}_{\mathsf{Restricted}}$.
\end{example}

\section{Soundness}\label{dSec:Soundness}
In this section, we prove the soundness of the algorithms in the style of \cite{CoPa12,Zh21b}. Since most of the rules in the algorithm is the same as in $\mathsf{hybrid}$-$\mathsf{ALBA}$ in \cite{ConRob}, therefore by dualizing their soundness proof, we get the soundness of the rules already in \cite{ConRob}. Therefore, we only treat the new rules in $\mathsf{ALBA}^{@}$.

\begin{theorem}[Soundness of the algorithm]\label{dThm:Soundness}
If $\mathsf{ALBA}^{@}$ runs successfully on $\phi\to\psi$ and outputs a first-order formula $\mathsf{FO}(\phi\leq\psi)$, then for any frame $\mathbb{F}=(W,R)$, $$\mathbb{F}\Vdash\phi\to\psi\mbox{ iff }\mathbb{F}\models\mathsf{FO}(\phi\to\psi).$$
\end{theorem}

\begin{proof}
The proof is similar to \cite[Theorem 8.1]{CoPa12}. Let $\phi\to\psi$ denote the input formula, $\nomi_0\leq\phi\ \&\ \psi\leq\neg\nomi_1\ \Rightarrow\ \nomi_0\leq\neg\nomi_1$ denote the system after the first-approximation rule, let $\textsf{Pure}(\phi\to\psi)$ denote the set of pure quasi-inequalities after Stage 2, let $\mathsf{FO(\phi\to\psi)}$ denote resulting first-order formula in Stage 3, then it suffices to show the equivalence from (\ref{dbCrct:Eqn0}) to (\ref{dbCrct:Eqn4}) given below:
\begin{eqnarray}
&&\mathbb{F}\Vdash\phi\to\psi\label{dbCrct:Eqn0}\\
&&\mathbb{F}\Vdash\nomi_0\leq\phi\ \&\ \psi\leq\neg\nomi_1\ \Rightarrow\ \nomi_0\leq\neg\nomi_1\label{dbCrct:Eqn2}\\
&&\mathbb{F}\Vdash\textsf{Pure}(\phi\to\psi)\label{dbCrct:Eqn3}\\
&&\mathbb{F}\vDash\mathsf{FO(\phi\to\psi)}\label{dbCrct:Eqn4}
\end{eqnarray}
The equivalence between (\ref{dbCrct:Eqn0}) and (\ref{dbCrct:Eqn2}) follows from the soundness of the first-approximation rule;

The equivalence between (\ref{dbCrct:Eqn2}) and (\ref{dbCrct:Eqn3}) follows from the soundness of the rules in Stage 2;

The equivalence between (\ref{dbCrct:Eqn3}) and (\ref{dbCrct:Eqn4}) follows from Proposition \ref{dProp:ST:ineq:quasi:mega}.
\end{proof}

In the remainder of the section, we prove the soundness of the rules in Stage 2.1, the splitting rules creating two systems in Stage 2.2, the approximation rule for $\to$ in Stage 2.2 and the residuation rules for $@$ in Stage 2.3.

For the soundness of a rule we mean that for any frame $\mathbb{F}$, the system before the application of a rule is valid in $\mathbb{F}$ iff the systems after the application of the rule is valid in $\mathbb{F}$. 

\begin{proposition}
The rules in Stage 2.1 are sound.
\end{proposition}

\begin{proof}
For the soundness of the rules in Stage 2.1, we only prove for (a), the other three rules are similar. 

By the validity of the equivalence $$(\alpha_1\land\beta\to\gamma)\land(\alpha_2\land\beta\to\gamma)\leftrightarrow((\alpha_1\lor\alpha_2)\land\beta\to\gamma),$$
it suffices to show that for any model $(\mathbb{F},V)$,
$$\mathbb{F},V\Vdash\nomi\leq\theta(@_{\nomj}\alpha)\mbox{ iff }(\mathbb{F},V\Vdash\nomi\leq\theta(\bot))\mbox{ or }(\mathbb{F},V\Vdash\nomi\leq\theta(\top)\mbox{ and }\mathbb{F},V\Vdash\nomj\leq\alpha).$$
Since $\mathbb{F},V\Vdash\nomi\leq\gamma$ iff $\mathbb{F},V,V(\nomi)\Vdash\gamma$, it suffices to show that
$$\mathbb{F},V,V(\nomi)\Vdash\theta(@_{\nomj}\alpha)\mbox{ iff }(\mathbb{F},V,V(\nomi)\Vdash\theta(\bot))\mbox{ or }(\mathbb{F},V,V(\nomi)\Vdash\theta(\top)\mbox{ and }\mathbb{F},V,V(\nomj)\Vdash\alpha).$$

$\Rightarrow:$ Assmue that $\mathbb{F},V,V(\nomi)\Vdash\theta(@_{\nomj}\alpha)$. Since $@_{\nomj}\alpha$ is either globally true or globally false, we have that there are two cases: $\mathbb{F},V\Vdash@_{\nomj}\alpha\leftrightarrow\bot$ or $\mathbb{F},V\Vdash@_{\nomj}\alpha\leftrightarrow\top$. In the first case we have $\mathbb{F},V,V(\nomi)\Vdash\theta(\bot)$, in the second case we have $\mathbb{F},V,V(\nomi)\Vdash\theta(\top)$ and $\mathbb{F},V\Vdash@_{\nomj}\alpha$, i.e.\ $\mathbb{F},V,V(\nomj)\Vdash\alpha$.

$\Leftarrow:$ Assume that $$(\mathbb{F},V,V(\nomi)\Vdash\theta(\bot))\mbox{ or }(\mathbb{F},V,V(\nomi)\Vdash\theta(\top)\mbox{ and }\mathbb{F},V,V(\nomj)\Vdash\alpha).$$
If $\mathbb{F},V,V(\nomi)\Vdash\theta(\bot)$ then by $+@_{\nomj}(\alpha)\prec+\theta$ we have that $\theta$ is monotone in the position of $@_{\nomj}\alpha$, so $\mathbb{F},V,V(\nomi)\Vdash\theta(@_{\nomj}\alpha)$. 

If $\mathbb{F},V,V(\nomi)\Vdash\theta(\top)\mbox{ and }\mathbb{F},V,V(\nomj)\Vdash\alpha$, then $\mathbb{F},V\Vdash @_{\nomj}\alpha$, therefore $\mathbb{F},V\Vdash @_{\nomj}\alpha\leftrightarrow\top$, so $\mathbb{F},V,V(\nomi)\Vdash\theta(@_{\nomj}\alpha)$.
\end{proof}

\begin{proposition}
The splitting rules creating two systems in Stage 2.2 are sound.
\end{proposition}

\begin{proof}
We only prove it for the splitting rule for $\lor$, the splitting rule for $\land$ is similar. By the validity of the equivalence $$(\alpha_1\land\beta\to\gamma)\land(\alpha_2\land\beta\to\gamma)\leftrightarrow((\alpha_1\lor\alpha_2)\land\beta\to\gamma),$$
it suffices to show that for any frame $\mathbb{F}$ and any valuation $V$ on $\mathbb{F}$,
$$\mathbb{F},V\Vdash\nomi\leq\beta\lor\gamma\mbox{ iff }(\mathbb{F},V\Vdash\nomi\leq\beta\mbox{ or }\mathbb{F},V\Vdash\nomi\leq\gamma),$$
which is equivalent to 
$$\mathbb{F},V,V(\nomi)\Vdash\beta\lor\gamma\mbox{ iff }(\mathbb{F},V,V(\nomi)\Vdash\beta\mbox{ or }\mathbb{F},V,V(\nomi)\Vdash\gamma),$$
which follows from the semantics.
\end{proof}

For the soundness of the approximation rule for $\to$ in Stage 2.2, see \cite{Zh21c}.

\begin{proposition}
The residuation rules for $@$ in Stage 2.3 are sound.
\end{proposition}

\begin{proof}
We only prove it for the residuation rule for $@$ where $@_{\nomj}\alpha$ is on the right-hand side, the other rule is similar.

By the validity of the equivalence $$(\alpha_1\land\beta\to\gamma)\land(\alpha_2\land\beta\to\gamma)\leftrightarrow((\alpha_1\lor\alpha_2)\land\beta\to\gamma),$$
it suffices to show that for any frame $\mathbb{F}$ and any valuation $V$ on $\mathbb{F}$,
$$\mathbb{F},V\Vdash\alpha\leq@_{\nomj}\beta\mbox{ iff }(\mathbb{F},V\Vdash\alpha\leq\bot\mbox{ or }\mathbb{F},V\Vdash\nomj\leq\beta),$$
which follows from the soundness of the ($@$-R-Res) rule in \cite{ConRob}.
\end{proof}
Since the algorithm $\mathsf{ALBA}^{@}_{\mathsf{Restricted}}$ is a restricted version of $\mathsf{ALBA}^{@}$, its soundness follows from the soundness of $\mathsf{ALBA}^{@}$.

\section{Success of $\mathsf{ALBA}^{@}$ and $\mathsf{ALBA}^{@}_{\mathsf{Restricted}}$}\label{dSec:Success}

In this section, we show that $\mathsf{ALBA}^{@}$ succeeds on all extended inductive formulas and that $\mathsf{ALBA}^{@}_{\mathsf{Restricted}}$ succeeds on all extended skeletal formulas. The proof is similar to \cite[Section 7]{Zh21c}, but we will stress on the special shape of the inequalities involved in the execution of $\mathsf{ALBA}^{@}_{\mathsf{Restricted}}$.

\begin{theorem}\label{dThm:Success}
$\mathsf{ALBA}^{@}$ succeeds on all extended inductive formulas and $\mathsf{ALBA}^{@}_{\mathsf{Restricted}}$ succeeds on all extended skeletal formulas.
\end{theorem}

\begin{lemma}\label{dLemma:Substage:2:1}
Given a system $\nomi_0\leq\phi\ \&\ \psi\leq\neg\nomi_1\ \Rightarrow\ \nomi_0\leq\neg\nomi_1$ obtained from Stage 1 where $+\phi$ and $-\psi$ are $(\Omega,\epsilon)$-extended inductive, by applying the rules in Stage 2.1 exhaustively, for each quasi-inequality $\mathsf{S}\ \Rightarrow\ \nomi_0\leq\neg\nomi_1$ obtained, the inequalities in $\mathsf{S}$ are in one of the following forms:
\begin{itemize}
\item $\nomk\leq\beta$, where $+\beta$ is $(\Omega,\epsilon)$-inductive;
\item $\beta\leq\neg\nomk$, where $-\beta$ is $(\Omega,\epsilon)$-inductive.
\end{itemize}
\end{lemma}

\begin{proof}
Suppose we have an inequality $\nomi\leq\theta(@_{\nomj}\alpha)$ where $+@_{\nomj}\alpha\prec+\theta$, $+\theta$ is $(\Omega,\epsilon)$-extended inductive and $@_{\nomj}$ is the farthest node in the $P_3$-part of a $\epsilon$-critical branch. Then by applying the decomposition rule in Stage 2.1, we get inequalities of the form $\nomi\leq\theta(\top)$, $\nomi\leq\theta(\bot)$ and $\nomj\leq\alpha$. For $\nomi\leq\theta(\top)$ and $\nomi\leq\theta(\bot)$, $+\theta(\top)$ and $+\theta(\bot)$ are still $(\Omega,\epsilon)$-extended inductive, but have one less $\epsilon$-critical branch with $P_3$-part, and in $\nomj\leq\alpha$, $+\alpha$ is $(\Omega,\epsilon)$-inductive, since it has no $P_3$-part in $\epsilon$-critical branches. 

For the other three cases where a decomposition rule can be applied, we get inequalities with either the $\theta(\top)$ and $\theta(\bot)$ part having one less $\epsilon$-critical branch with $P_3$-part, or the $\alpha$ part has no $P_3$-part in $\epsilon$-critical branches, and the two $\theta$ formulas with appropriate signs are still $(\Omega,\epsilon)$-extended inductive and the $\alpha$ formula is $(\Omega,\epsilon)$-inductive. 

By repeating the rules above, we get inequalities with no $P_3$-part on any $\epsilon$-critical branch, therefore all the $\epsilon$-critical branches in all formulas involved are with at most $P_1$ and $P_2$-part, therefore with appropriate signs, they are $(\Omega,\epsilon)$-inductive.
\end{proof}

\begin{definition}[$(\Omega,\epsilon)$-inner inductive signed generation tree]
Given an order-type $\epsilon$ and a dependence order $<_{\Omega}$, $*\in\{-,+\}$, the signed generation tree $*\phi$ of the formula $\phi(p_1,\ldots, p_n)$ is \emph{$(\Omega,\epsilon)$-inner inductive} if it is $(\Omega,\epsilon)$-extended inductive and the $P_3,P_2$-part on an $\epsilon$-critical branch are always empty, i.e.\ its $\epsilon$-critical branches have $P_1$-nodes only.
\end{definition}

\begin{lemma}\label{dLemma:Substage:2:2}
Given inequalities $\nomk\leq\beta$ and $\beta\leq\neg\nomk$ obtained from Stage 2.1 where $+\phi_i$ and $-\psi_i$ are $(\Omega,\epsilon)$-inductive, by applying the rules in Stage 2.2 exhaustively, for each quasi-inequality $\mathsf{S}\ \Rightarrow\ \nomi_0\leq\neg\nomi_1$ obtained, the inequalities that we get in $\mathsf{S}$ are in one of the following forms:
\begin{enumerate}
    \item pure inequalities without propositional variables;
    \item inequalities of the form $\nomi\leq\alpha$ where $+\alpha$ is $(\Omega,\epsilon)$-inner inductive;
    \item inequalities of the form $\beta\leq\neg\nomi$ where $-\beta$ is $(\Omega,\epsilon)$-inner inductive.
\end{enumerate}
\end{lemma}

\begin{proof}
Indeed, the rules in Stage 2.2 deal with $P_2$ nodes in the signed generation trees $+\beta$ and $-\beta$. For each rule, without loss of generality assume we start with an inequality of the form $\nomi\leq\alpha$, then by applying the rules in Stage 2.2, the new inequalities we get are either a pure inequality without propositional variables, or 
an inequality where the left-hand side (resp.\ right-hand side) is $\nomi$ (resp.\ $\neg\nomi$), and the other side is a formula $\alpha'$ which is a subformula of $\alpha$, such that $\alpha'$ has one root connective less than $\alpha$. Indeed, if $\alpha'$ is on the left-hand side (resp.\ right-hand side) then $-\alpha'$ ($+\alpha'$) is $(\Omega,\epsilon)$-inductive.

By applying the rules in Stage 2.2 exhaustively, we can eliminate all the $P_2$ connectives in the $\epsilon$-critical branches, so for non-pure inequalities, they become of form 2 or form 3.
\end{proof}

Now comes the difference between $\mathsf{ALBA}^{@}$ and $\mathsf{ALBA}^{@}_{\mathsf{Restricted}}$. We first deal with $\mathsf{ALBA}^{@}$.

\begin{lemma}\label{dLemma:Substage:2:3}
Assume we have an inequality $\nomi\leq\alpha$ or $\beta\leq\neg\nomi$ where $+\alpha$ and $-\beta$ are $(\Omega,\epsilon)$-inner inductive, by applying the rules in Stage 2.3, for each quasi-inequality $\mathsf{S}\ \Rightarrow\ \nomi_0\leq\neg\nomi_1$ obtained, the inequalities that we get in $\mathsf{S}$ are in one of the following forms:

\begin{enumerate}
\item $\alpha\leq p$, where $\epsilon(p)=1$, $-\alpha$ does not contain $p$ and is $\epsilon^{\partial}$-uniform;
\item $p\leq\beta$, where $\epsilon(p)=\partial$, $+\beta$ does not contain $p$ and is $\epsilon^{\partial}$-uniform;
\item $\gamma\leq\delta$, where $-\gamma,+\delta$ are $\epsilon^{\partial}$-uniform.
\end{enumerate}
\end{lemma}

\begin{proof}
For each of the inequality at the beginning of this stage, it is of the form $\nomi\leq\alpha$ or $\beta\leq\neg\nomi$, where 
\begin{itemize}
\item $+\alpha$ and $-\beta$ are $(\Omega,\epsilon)$-inner inductive;
\item $-\nomi$ and $+\neg\nomi$ are $\epsilon^{\partial}$-uniform;
\item for each $\epsilon$-critical occurrence $p$ in $+\alpha$ (resp.\ $-\beta$), $\nomi$ (resp.\ $\neg\nomi$) does not contain $p$.
\end{itemize}

By applying the splitting rules and the residuation rules in Stage 2.3 to an inequality $\gamma_1\leq\theta_1$ or $\theta_2\leq\gamma_2$ where 

\begin{itemize}
\item for each $\epsilon$-critical occurrence $p$ in $+\theta_1$ (resp.\ $-\theta_2$), $\gamma_1$ (resp.\ $\gamma_2$) does not contain $p$;
\item $-\gamma_1$ (resp.\ $+\gamma_2$) is $\epsilon^{\partial}$-uniform;
\item $+\theta_1$ (resp.\ $-\theta_2$) is $(\Omega,\epsilon)$-inner inductive;
\end{itemize}
it is easy to check that the $\epsilon^{\partial}$-uniform part will remain $\epsilon^{\partial}$-uniform and does not contain the relevant variables, and the other side is a formula $\theta'_i$ which is a subformula of $\theta_i$, such that $\theta'_i$ with appropriate sign is still $(\Omega,\epsilon)$-inner inductive but has one root connective less than $\theta_i$.

By applying these rules exhaustively, the $\epsilon^{\partial}$-uniform side will remain $\epsilon^{\partial}$-uniform and does not contain the relevant variables, the other side is either a propositional variable $p$ which, with appropriate sign, is $(\Omega,\epsilon)$-inner inductive, or is a  $\epsilon^{\partial}$-uniform formula (with appropriate sign, which is also $(\Omega,\epsilon)$-inner inductive since it contains no $\epsilon$-critical branch). Therefore, the inequality is of one of the three forms indicated.
\end{proof}

\begin{lemma}\label{dLemma:Substage:2:4}
Assume we have inequalities of the form as described in Lemma \ref{dLemma:Substage:2:3}, the Ackermann rules are applicable and therefore all propositional variables can be eliminated.
\end{lemma}

\begin{proof}
Immediate observation from the requirements of the Ackermann rules.
\end{proof}

\begin{proof}[Proof of Theorem \ref{dThm:Success} for $\mathsf{ALBA}^{@}$]
Assume we have an $(\Omega,\epsilon)$-extended inductive formula $\phi\to\psi$ as input. By Stage 1, we get the quasi-inequality $\nomi_0\leq\phi\ \&\ \psi\leq\neg\nomi_1\ \Rightarrow\ \nomi_0\leq\neg\nomi_1$. By Lemma \ref{dLemma:Substage:2:1}, \ref{dLemma:Substage:2:2}, \ref{dLemma:Substage:2:3}, we get quasi-inequalities as described there. Finally by Lemma \ref{dLemma:Substage:2:4}, the quasi-inequalities are in the right shape to apply the Ackermann rules, and thus we can eliminate all the propositional variables and the algorithm succeeds on the input.
\end{proof}

For $\mathsf{ALBA}^{@}_{\mathsf{Restricted}}$, in Lemma \ref{dLemma:Substage:2:2}, the critical branches in $(\Omega,\epsilon)$-inner inductive formulas are without $P_1$-part, so they are either $\epsilon^{\partial}$-uniform, or they are already propositional variables, therefore the result for Lemma \ref{dLemma:Substage:2:3} is automatically satisfied, so by Lemma \ref{dLemma:Substage:2:4}, $\mathsf{ALBA}^{@}_{\mathsf{Restricted}}$ succeeds on all extended skeletal formulas.

\section{Completeness results}\label{dSec:Completeness}
In this section, we will prove that given any extended skeletal formula $\phi\to\psi$, the logic $\mathbf{K}_{\mathcal{H}(@)}+(\phi\to\psi)$ is sound and strongly complete with respect to the class of frames defined by $\phi\to\psi$. The strategy is the same as \cite{Zh22b} and the presentation is similar.

Our proof strategy is as follows:

\begin{itemize}
\item We translate of each system in $\textsf{Pure}(\phi\to\psi)$ into $\mathcal{L}(@)$-formulas, which results in a set $\Pi$ of $\mathcal{L}(@)$-formulas, and we prove that $\phi\to\psi$ and $\Pi$ define the same class of frames.
\item We then prove that each $\pi\in\Pi$ is provable in $\mathbf{K}_{\mathcal{H}(@)}+(\phi\to\psi)$. Therefore, by the Theorem \ref{dCompleteness:Pure}, we get the soundness and strong completeness of $\mathbf{K}_{\mathcal{H}(@)}+(\phi\to\psi)$.
\end{itemize}

\subsection{The translation}

As one can easily observe, in $\mathsf{ALBA}^{@}_{\mathsf{Restricted}}$, in the systems obtained in Stage 2, for each inequality in them, either the left-hand side is $\nomi$, or the right-hand side is $\neg\nomi$. Indeed, the inequality $\nomi\leq\gamma$ is equivalent to the $\mathcal{L}(@)$-formula $@_\nomi \gamma$ , and the inequality $\gamma\leq\neg\nomi$ is equivalent to the $\mathcal{L}(@)$-formula $\neg @_\nomi \gamma$. Therefore, the systems obtained in Stage 2 are equivalent to $\mathcal{L}(@)$-formulas.

\begin{definition}[Translation into $\mathcal{L}(@)$-formulas]
We define the translation of the inequalities of the form $\nomi\leq\gamma$, $\gamma\leq\neg\nomi$ into $\mathcal{L}(@)$-formulas as follows:
\begin{itemize}
\item $\mathsf{Tr}(\nomi\leq\gamma):=@_\nomi \gamma$;
\item $\mathsf{Tr}(\gamma\leq\neg\nomi):=\neg @_\nomi \gamma$.
\end{itemize}
When an inequality is of both of the forms above, we can take any of the two since they are equivalent.

Given a quasi-inequality $\mathsf{Quasi}$ of the form $\mathsf{Ineq}_1\ \&\ \ldots\ \&\ \mathsf{Ineq}_n\ \Rightarrow\ \nomi\leq\neg\nomj$ where each of $\mathsf{Ineq}_1, \ldots, \mathsf{Ineq}_n$ is of the form $\nomi\leq\gamma$ or $\gamma\leq\neg\nomi$, define  
$$\mathsf{Tr}(\mathsf{Quasi}):=\mathsf{Tr}(\mathsf{Ineq}_1)\land\ldots\land\mathsf{Tr}(\mathsf{Ineq}_n)\to\neg @_{\nomi}\nomj$$

Given a set $\mathsf{QuasiSet}$ of quasi-inequalities of the form above, define $$\mathsf{Tr}(\mathsf{QuasiSet}):=\bigwedge_{\mathsf{Quasi}\in\mathsf{QuasiSet}}\mathsf{Tr}(\mathsf{Quasi}).$$
\end{definition}

\begin{proposition}\label{dProp:Translation}
For each inequality $\mathsf{Ineq}$ of the form $\nomi\leq\gamma$ or $\gamma\leq\neg\nomi$, each quasi-inequality $\mathsf{Quasi}$ of the form described above, each set $\mathsf{QuasiSet}$ of quasi-inequalities of the form above, we have that for any model $\mathbb{M}$, 
$$\mathbb{M}\Vdash\mathsf{Ineq}\mbox{ iff }\mathbb{M}\Vdash\mathsf{Tr}(\mathsf{Ineq})$$
$$\mathbb{M}\Vdash\mathsf{Quasi}\mbox{ iff }\mathbb{M}\Vdash\mathsf{Tr}(\mathsf{Quasi})$$
$$\mathbb{M}\Vdash\mathsf{QuasiSet}\mbox{ iff }\mathbb{M}\Vdash\mathsf{Tr}(\mathsf{QuasiSet}).$$
\end{proposition}

\subsection{Provability of the translations}

\begin{lemma}[Lemma 7.5 in \cite{Zh22b}]
For the quasi-inequality $\mathsf{Quasi}:=\nomi_0\leq\phi\ \&\ \psi\leq\neg\nomi_1\ \Rightarrow\ \nomi_0\leq\neg\nomi_1$, then we have that $\vdash_{\phi\to\psi}\mathsf{Tr}(\mathsf{Quasi})$.
\end{lemma}

Now we will prove that for each system $\mathsf{S}\ \Rightarrow\ \nomi_0\leq\neg\nomi_1$ obtained during Stage 2, $\mathsf{Tr}(\mathsf{S}\ \Rightarrow\ \nomi_0\leq\neg\nomi_1)$ is provable.

\begin{lemma}
Given the quasi-inequality $\nomi_0\leq\phi\ \&\ \psi\leq\neg\nomi_1\ \Rightarrow\ \nomi_0\leq\neg\nomi_1$ obtained in Stage 1, for each system $\mathsf{S}\ \Rightarrow\ \nomi_0\leq\neg\nomi_1$ obtained during Stage 2, $\vdash_{\phi\to\psi}\mathsf{Tr}(\mathsf{S}\ \Rightarrow\ \nomi_0\leq\neg\nomi_1)$.
\end{lemma}

\begin{proof}
First of all, since in each inequality in the system, either the left-hand side is a nominal, or the right-hand side is the negation of a nominal, $\mathsf{S}\ \Rightarrow\ \nomi_0\leq\neg\nomi_1$ can be translated.

We prove by induction on the algorithm steps in Stage 2 that for each system $\mathsf{S}\ \Rightarrow\ \nomi_0\leq\neg\nomi_1$ obtained during Stage 2, $\vdash_{\phi\to\psi}\mathsf{Tr}(\mathsf{S}\ \Rightarrow\ \nomi_0\leq\neg\nomi_1)$ is provable.

\begin{itemize}
\item For the basic step, obviously $\vdash_{\phi\to\psi}\mathsf{Tr}(\nomi_0\leq\phi\ \&\ \psi\leq\neg\nomi_1\ \Rightarrow\ \nomi_0\leq\neg\nomi_1)$.
\item For Stage 2.1, for the rule (a), it suffices to show that from $\vdash_{\phi\to\psi}@_{\nomi}\theta(@_{\nomj}\alpha)\land\gamma\to\delta$ one can get $\vdash_{\phi\to\psi}@_{\nomi}\theta(\bot)\land\gamma\to\delta$ and 
$\vdash_{\phi\to\psi}@_{\nomi}\theta(\top)\land@_{\nomj}\alpha\land\gamma\to\delta$, which follows from $\vdash_{\Sigma}@_{\nomi}\theta(@_{\nomj}\alpha)\leftrightarrow@_{\nomi}\theta(\bot)\lor(@_{\nomi}\theta(\top)\land@_{\nomj}\alpha)$
on page \pageref{dpage:decomposition:Equivalence}.

For (b),(c),(d), the proofs are similar.
\item For the splitting rules involving two systems, we show it for the rule involving $\lor$. It suffices to show that from $\vdash_{\phi\to\psi}@_{\nomi}(\beta\lor\gamma)\land\theta\to\delta$ one can get $\vdash_{\phi\to\psi}@_{\nomi}\beta\land\theta\to\delta$ and 
$\vdash_{\phi\to\psi}@_{\nomi}\gamma\land\theta\to\delta$, which follows from $\vdash_{\phi\to\psi}@_{\nomi}(\beta\lor\gamma)\leftrightarrow(@_{\nomi}\beta\lor@_{\nomi}\gamma)$.

For the rule involving $\land$, the proof is similar.
\item For the splitting rules within a single system, the approximation rules for $\Diamond,\Box,@,\to$, the residuation rules for $\neg$, the proof is similar to \cite[Lemma 7.6]{Zh22b}.
\item In $\mathsf{ALBA}^{@}_{\mathsf{Restricted}}$, there is no Stage 2.3, so this part can be omitted. 
\item For the Ackermann rules, the proof is similar to \cite[Lemma 7.6]{Zh22b}.
\end{itemize}
\end{proof}

\begin{corollary}\label{dCor:Main}
Given an extended skeletal formula $\phi\to\psi$, for each quasi-inequality $\mathsf{Quasi}$ in $\textsf{Pure}(\phi\to\psi)$, we have that $\vdash_{\phi\to\psi}\mathsf{Tr}(\mathsf{Quasi})$, therefore $\vdash_{\phi\to\psi}\mathsf{Tr}(\textsf{Pure}(\phi\to\psi))$.
\end{corollary}

\subsection{Main Proof}

\begin{theorem}[Similar to Theorem 7.8 in \cite{Zh22b}]
For any extended skeletal formula $\phi\to\psi$, $\mathbf{K}_{\mathcal{H}(@)}+(\phi\to\psi)$ is sound and strongly complete with respect to the class of frames $\mathcal{F}$ defined by $\phi\to\psi$.
\end{theorem}

\begin{proof}
The proof is the same as \cite[Theorem 7.8]{Zh22b}. We prove that for any set $\Gamma$ of $\mathcal{L}(@)$-formulas and any $\mathcal{L}(@)$-formula $\gamma$, the following three conditions are equivalent:

\begin{enumerate}
\item $\Gamma\vdash_{\phi\to\psi}\gamma$;
\item $\Gamma\Vdash_{\mathcal{F}}\gamma$;
\item $\Gamma\vdash_{\mathsf{Tr}(\mathsf{Pure}(\phi\to\psi))}\gamma.$
\end{enumerate}

\begin{itemize}
\item[(1$\Rightarrow$2):] The soundness proof is easy. 
\item[(2$\Rightarrow$3):] From 
\begin{center}
\begin{tabular}{l l l}
& $\mathbb{F}\Vdash\phi\to\psi$ &\\
iff & $\mathbb{F}\Vdash\mathsf{Pure}(\phi\to\psi)$ & (Theorem \ref{dThm:Soundness})\\
iff & $\mathbb{F}\Vdash\mathsf{Tr}(\mathsf{Pure}(\phi\to\psi))$, & (corollary of Proposition \ref{dProp:Translation})\\
\end{tabular}
\end{center}
$\mathcal{F}$ is also defined by $\mathsf{Tr}(\mathsf{Pure}(\phi\to\psi))$. By Theorem \ref{dCompleteness:Pure}, we have the completeness of $\mathbf{K}_{\mathcal{H}(@)}+\mathsf{Tr}(\mathsf{Pure}(\phi\to\psi))$ with respect to $\mathcal{F}$.
\item[(3$\Rightarrow$1):] By Corollary \ref{dCor:Main}, $\vdash_{\phi\to\psi}\mathsf{Tr}(\mathsf{Pure}(\phi\to\psi))$, therefore all theorems of $\mathbf{K}_{\mathcal{H}(@)}+\mathsf{Tr}(\mathsf{Pure}(\phi\to\psi))$ are also theorems of $\mathbf{K}_{\mathcal{H}(@)}+(\phi\to\psi)$.
\end{itemize}
\end{proof}
The following corollary follows from an easy adaptation of the previous results to a set $\Sigma$ of extended skeletal formulas:

\begin{corollary}
For any set $\Sigma$ of extended skeletal formulas, $\mathbf{K}_{\mathcal{H}(@)}+\Sigma$ is sound and strongly complete with respect to the class of frames $\mathcal{F}$ defined by $\Sigma$.
\end{corollary}

\paragraph{Acknowledgement} The research of the author is supported by the Taishan Young Scholars Program of the Government of Shandong Province, China (No.tsqn201909151).

\bibliographystyle{abbrv}
\bibliography{McKinsey}

\begin{thebibliography}{10}

\bibitem{BetCMaVi04}
N.~Bezhanishvili, B.~ten Cate, M.~Marx, and P.~Viana.
\newblock Sahlqvist theory and transfer results for hybrid logics.
\newblock In R.~Schmidt, I.~Pratt-Hartmann, M.~Reynolds, and H.~Wansing,
  editors, {\em Preliminary proceedings of Advances in Modal Logic 2004}, 2004.

\bibitem{tCBl06}
P.~Blackburn and B.~ten Cate.
\newblock Pure extensions, proof rules, and hybrid axiomatics.
\newblock {\em Studia Logica}, 84(2):277--322, 2006.

\bibitem{BeBlWo06}
P.~Blackburn, J.~van Benthem, and F.~Wolter.
\newblock {\em Handbook of modal logic}, volume~3.
\newblock Elsevier, 2006.

\bibitem{Co09}
W.~Conradie.
\newblock Completeness and correspondence in hybrid logic via an extension of
  {SQEMA}.
\newblock {\em Electronic Notes in Theoretical Computer Science}, 231:175 --
  190, 2009.
\newblock Proceedings of the 5th Workshop on Methods for Modalities (M4M5
  2007).

\bibitem{CoGoVa06b}
W.~Conradie, V.~Goranko, and D.~Vakarelov.
\newblock Algorithmic correspondence and completeness in modal logic. {II}.
  {P}olyadic and hybrid extensions of the algorithm {SQEMA}.
\newblock {\em Journal of Logic and Computation}, 16(5):579--612, 09 2006.

\bibitem{CoPa12}
W.~Conradie and A.~Palmigiano.
\newblock Algorithmic correspondence and canonicity for distributive modal
  logic.
\newblock {\em Annals of Pure and Applied Logic}, 163(3):338 -- 376, 2012.

\bibitem{CPZ:Trans}
W.~Conradie, A.~Palmigiano, and Z.~Zhao.
\newblock Sahlqvist via translation.
\newblock {\em {Logical Methods in Computer Science}}, {Volume 15, Issue 1},
  Feb. 2019.

\bibitem{ConRob}
W.~Conradie and C.~Robinson.
\newblock On {S}ahlqvist theory for hybrid logic.
\newblock {\em Journal of Logic and Computation}, 27(3):867--900, 2017.

\bibitem{GaGo93}
G.~Gargov and V.~Goranko.
\newblock Modal logic with names.
\newblock {\em Journal of Philosophical Logic}, 22(6):607--636, 1993.

\bibitem{GoVa01}
V.~Goranko and D.~Vakarelov.
\newblock Sahlqvist formulas in hybrid polyadic modal logics.
\newblock {\em Journal of Logic and Computation}, 11(5):737--754, 10 2001.

\bibitem{Ho06}
I.~Hodkinson.
\newblock Hybrid formulas and elementarily generated modal logics.
\newblock {\em Notre Dame Journal of Formal Logic}, 47(4):443--478, 10 2006.

\bibitem{HoPa10}
I.~Hodkinson and L.~Paternault.
\newblock Axiomatizing hybrid logic using modal logic.
\newblock {\em Journal of Applied Logic}, 8(4):386 -- 396, 2010.
\newblock Special Issue on Hybrid Logics.

\bibitem{PaSoZh16}
A.~Palmigiano, S.~Sourabh, and Z.~Zhao.
\newblock Sahlqvist theory for impossible worlds.
\newblock {\em Journal of Logic and Computation}, 27(3):775--816, 2017.

\bibitem{Sa75}
H.~Sahlqvist.
\newblock Completeness and correspondence in the first and second order
  semantics for modal logic.
\newblock In {\em Studies in Logic and the Foundations of Mathematics},
  volume~82, pages 110--143. 1975.

\bibitem{Ta05}
K.~Tamura.
\newblock Hybrid logic with pure and {S}ahlqvist axioms.
\newblock {\em http://www.st.nanzan-u.ac.jp/info/sasaki/2005mlg/43-45.pdf.}

\bibitem{tC05}
B.~ten Cate.
\newblock {\em Model theory for extended modal languages}.
\newblock PhD thesis, University of Amsterdam, 2005.

\bibitem{tCMaVi06}
B.~ten Cate, M.~Marx, and J.~P. Viana.
\newblock Hybrid logics with {S}ahlqvist axioms.
\newblock {\em Logic Journal of the IGPL}, (3):293--300, 2006.

\bibitem{vB83}
J.~van Benthem.
\newblock {\em Modal logic and classical logic}.
\newblock Bibliopolis, 1983.

\bibitem{Wa07}
H.~Wassenaar.
\newblock Taming {F}rankenstein's logic: How {I} turned the tables on hybrid
  logic.
\newblock Master's thesis, University of Groningen, the Netherlands, 2007.

\bibitem{Zh21b}
Z.~Zhao.
\newblock {Algorithmic correspondence and canonicity for possibility
  semantics}.
\newblock {\em Journal of Logic and Computation}, 31(2):523--572, 2021.

\bibitem{Zh21c}
Z.~Zhao.
\newblock Algorithmic correspondence for hybrid logic with binder.
\newblock {\em Logic Journal of the IGPL}, 2021.

\bibitem{Zh22b}
Z.~Zhao.
\newblock Sahlqvist-type completeness theory for hybrid logic with binder.
\newblock {\em submitted. ArXiv preprint arXiv:2207.01288}, 2022.

\end{thebibliography}

\end{document}